\documentclass[10pt,reqno]{amsart}
\usepackage[T1]{fontenc}
\usepackage[utf8]{inputenc}
\usepackage{lmodern}
\usepackage{amsmath,amssymb,mathtools}
\usepackage{booktabs,array}
\usepackage{microtype}
\usepackage{xcolor}
\usepackage[hidelinks]{hyperref}
\usepackage{enumitem}
\setlist{nosep}

\newtheorem{theorem}{Theorem}[section]

\newtheorem{conjecture}[theorem]{Conjecture}
\newtheorem{lemma}[theorem]{Lemma}
\newtheorem{corollary}[theorem]{Corollary}

\newcommand{\src}{\mathrm{src}}
\newcommand{\tgt}{\mathrm{tgt}}
\DeclareRobustCommand{\rev}[1]{#1}

\makeatletter
\@removefromreset{equation}{section}

\renewcommand{\tagform@}[1]{\maketag@@@{(\ignorespaces#1\unskip\@@italiccorr)}}
\let\ams@setauthors\@setauthors
\renewcommand{\@setauthors}{%
  \begingroup
  \let\MakeUppercase\@firstofone
  \ams@setauthors
  \endgroup
}
\makeatother

\title{On \rev{the} Erd\H{o}s Five-Edge Intersection Problem}
\author[Chengrui Fang and Jianfeng Hou]{Chengrui Fang\rev{, }Jianfeng Hou\\[4pt]
\normalfont\footnotesize Center for Discrete Mathematics, Fuzhou University, Fuzhou, China}
\date{}

\begin{document}

\begin{abstract}
For an $n$-vertex graph $G$ and a permutation $\pi$ of its vertex set, let
\[
 I_G(\pi)=|E(G)\cap E(G_{\pi})|,\qquad \mu(G)=\min_{\pi} I_G(\pi),
\]
where $G_{\pi}$ is the copy of $G$ obtained by relabelling every vertex $x\in V(G)$ as $\pi(x)$. 
Let $f(n,k)$ be the minimum number of edges in an $n$-vertex graph $G$ satisfying $\mu(G)\ge k$.
Erd\H{o}s recorded a construction of Mullin showing $f(n,5)\le 2n-2$ and asked whether equality holds for sufficiently large $n$.
We prove that it does:
\[
 f(n,5)=2n-2
\]
for all sufficiently large $n$.

The proof  strategy is a core--buffer--completion framework: it moves the few high-degree vertices  into carefully chosen low-degree positions, confines the allowed overlap to this bounded part, and then relabels the sparse remainder without creating any additional common edge.
\end{abstract}


\maketitle

\section{Introduction}

Let $G$ be a finite simple graph on vertex set $V$. For a permutation $\pi$ of $V$, \rev{let $G_{\pi}$ be} the copy of $G$ obtained by relabelling every vertex $x\in V$ as $\pi(x)$. More precisely,  $G_{\pi}$ is the graph on the same vertex set $V$ with edge set 
\[
E(G_{\pi})=\{\,\pi(x)\pi(y):xy\in E(G)\,\}.
\]
Define 
\[
 I_G(\pi):=|E(G)\cap E(G_{\pi})|,\qquad
 \mu(G):=\min_{\pi} I_G(\pi), 
\]
where  the minimum is taken over all permutations of  $V$. For an integer $k\ge1$, define
\[
 f(n,k):=\min\{\,|E(G)|: |V(G)|=n,\ \mu(G)\ge k\,\}.
\]
Erd\H{o}s~\cite{Erdos} recorded this parameter and asked to determine or estimate $f(n,k)$. He observed, by a simple averaging argument, that
\begin{equation}\label{eq:f(n,k)-Erdos-bounds}
 k^{1/2}n-c\le f(n,k)\le4k^{1/2}n
\end{equation}
for an absolute constant $c$, and for sufficiently large $k$, 
\[
 f(n,k)=(2+o(1))k^{1/2}n. 
\]

In the same paper, Erd\H{o}s \cite{Erdos} observed that the problem of determining $f(n,k)$ becomes particularly interesting when $k$ is small, and  established the base case $f(n,1)=n-1$. He also noted that Chung, Graham, and Murty had shown $f(n,2)=f(n,3)=\lfloor 3n/2 \rfloor$. For $k=4$, the complete bipartite graph $K_{2, n-2}$ provides the upper bound $f(n,4)\le 2n-4$, which, in conjunction with \eqref{eq:f(n,k)-Erdos-bounds}, implies $f(n,4)=2n-O(1)$. Turning to $k=5$, Erd\H{o}s mentioned a construction due to Mullin: take a cycle on $n-1$ vertices and join the remaining vertex to every vertex of the cycle. This wheel-type graph has $2n-2$ edges and yields $f(n,5)\le 2n-2$, prompting Erd\H{o}s to conjecture that equality might hold.
\begin{conjecture}\label{Conj:Erdos-conj-f(n,5)}
  $f(n,5)= 2n-2$. 
\end{conjecture}
\rev{\.{Z}ak's} general bounds~\cite[Lemma~16 and Theorem~17]{Zak} \rev{imply} that for $n>6$
\begin{equation*}
 \sqrt{2n(n-1)}+1<f(n,5)\le 2n-2.
\end{equation*}
In this paper, \rev{we confirm} Conjecture \ref{Conj:Erdos-conj-f(n,5)} for sufficiently large $n$. 
\begin{theorem}\label{thm:main}
There exists $n_0$ such that for every $n\ge n_0$,
\[
 f(n,5)=2n-2.
\]
\end{theorem}

Note that the upper bound is easy. For the lower bound, the proof strategy is to prove the contrapositive: every graph $G$ with at most $2n-3$ edges can be relabelled so that at most four edges survive.  The main difficulty is to handle high-degree vertices in $G$. 
Our first idea is to separate these few high-degree vertices from the rest of the graph and map them into carefully chosen low-degree buffer vertices.  The second idea is to encode every remaining possible common edge as a forbidden image in a list-packing problem and then complete the relabelling by a sparse random-permutation  argument. 

The paper is organized as follows. Section~2 presents the extremal construction and packing tools. Section~3 excludes almost-universal vertices. Section~4 proves the main theorem via the core--buffer--completion method. Section~5 gives concluding remarks.

\section{Preliminaries}\label{Preliminaries}
In this section, we list some known results and introduce the list packing  that will be used in our proof.

Let $G$ be a graph. We denote by $e(G) = |E(G)|$ its number of edges. For a vertex $v \in V(G)$, let $N_G(v)$ denote the \emph{neighbourhood} of $v$ in $G$, and let $d_G(v) = |N_G(v)|$ be its \emph{degree}. Then $v$ is \emph{universal} if $d_G(v)=|V(G)|-1$, and \rev{we call $v$ a $k$-vertex} if $d_G(v)=k$. 
Write $\Delta(G)$ and $\delta(G)$ for the maximum and minimum degrees of $G$, respectively. For disjoint subsets $X, Y \subseteq V(G)$, we define $N_G(X) = \bigcup_{x \in X} N_G(x)$, and for a vertex $v$, write $d_X(v) = |N_G(v) \cap X|$. Let $E_G(X, Y)$ denote the set of edges with one endpoint in $X$ and the other in $Y$, and let $e_G(X,Y)=|E_G(X, Y)|$.  An edge $e\in E_G(X, Y)$ is called a \emph{$X$--$Y$ edge}.  A set $W\subseteq V(G)$ is called an \emph{independent set} if no two vertices of $W$ are adjacent in $G$.
The \emph{independence number} of $G$, denoted by $\alpha(G)$, is the maximum size of an independent set in $G$. When the underlying graph is clear from the context, we shall omit the subscript $G$ for brevity. 

The following standard greedy bound for independent sets will be used several times; see, for example, Diestel~\cite{Diestel}.

\begin{theorem}\label{thm:greedy-alpha}
Every finite graph $G$ satisfies
\[
 \alpha(G)\ge\frac{|V(G)|}{\Delta(G)+1}.
\]
\end{theorem}

We shall use the following classical value, which Erd\H{o}s~\cite{Erdos} attributes to Chung, Graham, and Murty\rev{.}
\begin{theorem}[Chung, Graham and Murty]\label{CGM-THM-f-N,2}
  $f(N,2)=\left\lfloor\frac{3N}{2}\right\rfloor$\rev{.}  
\end{theorem}

For the sake of completeness, we include a proof of the upper bound on $f(n,5)$.

\begin{lemma}\label{prop:upper}
For every $n\ge6$, $f(n,5)\le2n-2$\rev{.}
\end{lemma}

\begin{proof}
Taking a $2$-regular graph $H$ on $n-1$ vertices and joining the remaining vertex $x$ to every vertex of $H$ yields a new graph $G$. It suffices to show that $I_G(\pi) \ge 5$ for every permutation $\pi$ of $V(G)$. Indeed, let $x' := \pi(x)$ be  the universal vertex of $G_\pi$.  If $x = x'$, then the $n-1$ edges incident to $x$ are common to both $G$ and $G_\pi$, giving $I_G(\pi) \ge n-1 \ge 5$ for $n \ge 6$. Now suppose $x \ne x'$. In $G$, the vertex $x'$ has degree three, whereas in $G_\pi$ it is universal; hence all three $G$-edges incident to $x'$ are common. Symmetrically, all three $G_\pi$-edges incident to $x$ are also common in $G$. The edge $xx'$ is counted twice, so the total number of distinct common edges is at least $3 + 3 - 1 = 5$.
\end{proof}

Now we  recall the notion of list packing.  For two graphs $G_1=(V_1,E_1)$ and $G_2=(V_2,E_2)$ with $|V_1|=|V_2|$,  a \emph{packing} of $G_1$ with $G_2$ is a bijection $\phi:V_1\to V_2$ such that $\phi(u)\phi(v)\notin E_2$ for every $uv\in E_1$.  We call $V_1$ the \emph{source} side and $V_2$ the \emph{target} side.  When $G_1=G_2=G$, such a bijection is a \emph{self-packing} of $G$.  A graph triple $(G_1,G_2,F)$ consists of two $n$-vertex graphs $G_i=(V_i,E_i)$ and a bipartite graph $F$ with bipartition $(V_1,V_2)$. We shall call $G_1$ and $G_2$ the white graphs, and $F$ the yellow graph of forbidden pairs.
A bijection $\phi:V_1\to V_2$ is a \emph{list packing} if
\[
 uv\in E_1\Longrightarrow \phi(u)\phi(v)\notin E_2,
 \qquad u\phi(u)\notin E(F)\quad(u\in V_1).
\]
We remark that the edges of $F$ are the forbidden vertex-image pairs.

The following is the list version of the Sauer--Spencer theorem~\cite{SS} due to Gy\H{o}ri, Kostochka, McConvey and Yager~\cite{GKMY}, which is a key ingredient in our proof..

\begin{theorem}[Gy\H{o}ri, Kostochka, McConvey and Yager~\cite{GKMY}]\label{thm:list}
Let $(G_1,G_2,F)$ be \rev{a} graph triple with $G_i=(V_i,E_i)$ for $i=1,2$.  If 
\[
 \Delta(G_1)\Delta(G_2)+\Delta(F)<|V_1|/2,
\]
then $(G_1,G_2,F)$ has a list packing.
\end{theorem}

We shall repeatedly use the following elementary interpretation.
Suppose a partial bijection $\psi:S\to T$ has already been prescribed between two copies of a graph $G$, and suppose no edge with both ends in $S$ is sent to an edge with both ends in $T$.
For the remaining source vertices $u$ and target vertices $v$, declare $uv$ forbidden if there is $s\in S$ such that
\[
 su\in E(G),\qquad \psi(s)v\in E(G).
\]
Any list packing of the two remaining white graphs with these forbidden pairs extends $\psi$ to a packing of the two copies of $G$.

We present a list packing result in the setting where the maximum degree of the three graphs is small, yet still linear in the order. To prepare the ground, we first recall the classical Lov\'{a}sz Local Lemma (see, e.g., Alon and Spencer~\cite[Chapter~5]{AS}).

\begin{lemma}[Asymmetric Lopsided Local Lemma]\label{lem:lll}
Let $\mathcal A$ be a finite family of events in a probability space.  Suppose that $\Gamma$ is a negative dependency graph on the vertex set $\mathcal A$, meaning that for every $A\in\mathcal A$ and every set $\mathcal S\subseteq\mathcal A\setminus(\{A\}\cup N_{\Gamma}(A))$, where $N_{\Gamma}(A)$ denotes the neighbours of $A$ in $\Gamma$,
\[
\Pr\left(A\,\middle|\,\bigcap_{B\in\mathcal S}\overline B\right)\le\Pr(A)
\]
whenever the conditional probability is defined.  If there are numbers $x_A\in[0,1)$ such that
\[
\Pr(A)\le x_A\prod_{B\in N_{\Gamma}(A)}(1-x_B)
\]
for every $A\in\mathcal A$, then
\[
\Pr\left(\bigcap_{A\in\mathcal A}\overline A\right)>0.
\]
\end{lemma}

\rev{The following result} follows from a standard application of Lemma \ref{lem:lll}.

\begin{lemma}\label{lem:sparse}
\rev{There exist an absolute constant $\varepsilon_0>0$ and an integer $N_0>0$ such that the following statement holds.}
Let $(G_1,G_2,F)$ be a graph triple with $|V(G_1)|=|V(G_2)|=n\ge N_0$, 
\[
 e(G_1),e(G_2)\le3n,
\]
and
\[
 \max\{\Delta(G_1),\Delta(G_2),\Delta(F)\}\le\varepsilon_0n.
\]
Then $(G_1,G_2,F)$ has a list packing.
\end{lemma}

\begin{proof}
Set $\varepsilon_0:=10^{-3}$ and $N_0:=10^3$. Let $(G_1,G_2,F)$ be a graph triple with $|V(G_1)|=|V(G_2)|=n\ge N_0$, $e(G_1),e(G_2)\le3n$, and $\max\{\Delta(G_1),\Delta(G_2),\Delta(F)\}\le\varepsilon_0n$\rev{.}
Choose a  bijection $\phi:V(G_1)\to V(G_2)$ randomly and uniformly. Then $\phi$ fails to be a list packing only if it maps some source edge to a target edge, or maps some source vertex to a forbidden target. For each source edge $u_1u_2\in E(G_1)$ \rev{and each} target edge $v_1v_2\in E(G_2)$, define the white bad event $W_{u_1u_2,v_1v_2}$ to be 
\[
 \phi(u_1)=v_1,\qquad \phi(u_2)=v_2.
\]
Clearly, the probability that $W_{u_1u_2,v_1v_2}$ happens \rev{is}
\[
p_W=\frac1{n(n-1)}.
\]
For every forbidden pair $uv\in E(F)$, define the yellow bad event $Y_{uv}$ to be $\phi(u)=v$. Then the probability that $Y_{uv}$ happens is 
\[
 p_F=\frac1n.
\]

Two canonical events are called conflicting when their prescribed partial injections are incompatible. Put
\[
 D=\max\{\Delta(G_1),\Delta(G_2)\},\qquad L=\Delta(F).
\]
Then $D,L\le \varepsilon_0 n$. 
Note that a white bad event $W_{u_1u_2,v_1v_2}$ uses two source vertices $u_1, u_2$, and two target vertices $v_1, v_2$.
Conflicts through $u_1, u_2$ contribute at most $4\Delta(G_1)e(G_2)$ white events.
The analogous target-side count is at most $4\Delta(G_2)e(G_1)$.
Thus $W_{u_1u_2,v_1v_2}$ conflicts with at most
\[
 4\Delta(G_1)e(G_2)+4\Delta(G_2)e(G_1)\le24Dn
\]
white events and at most $4L$ yellow events.
Similarly, a yellow  event $Y_{uv}$ conflicts through $u$ or $v$ with at most
\[
 2\Delta(G_1)e(G_2)+2\Delta(G_2)e(G_1)\le12Dn
\]
white events and with at most $2L$ yellow events.

By Lemma~\ref{lem:lll}, it is enough to find numbers $x_E\in(0,1)$ such that
\begin{equation}\label{EQ:main-equation-LLL}
     \Pr(E)\le x_E\prod_{E'\in N_{\tau} (E)}(1-x_{E'})
\end{equation}
for every bad event $E$. For \rev{a} white bad event $W:=W_{u_1u_2,v_1v_2}$, let 
\[
x_W=2p_W=\frac2{n(n-1)}, 
\]
and for the yellow bad event $Y_{uv}$, let 
\[
x_{Y_{uv}}=2\rev{p_F}=\frac2n. 
\]
Then 
\[
\sum_{E'\in N_{\tau}(W)}\rev{x_{E'}}\le  24Dn\frac2{n(n-1)}+4L\frac2n
 =\frac{48D}{n-1}+\frac{8L}{n}<0.1, 
\]
and 
\[
\sum_{E'\in N_{\tau}(F_{uv})}\rev{x_{E'}}\le   12Dn\frac2{n(n-1)}+2L\frac2n
 =\frac{24D}{n-1}+\frac{4L}{n}<0.1.
\]
\rev{By the choice of $\varepsilon_0$ and $N_0$, both sums are less than $0.1$. Together with the inequality $\prod_i(1-x_i)\ge1-\sum_i x_i$, this shows that \eqref{EQ:main-equation-LLL} holds and completes the proof.}
\end{proof}

\section{Almost-universal vertices are impossible}

In this section, we show that a counterexample to Theorem \ref{thm:main} cannot have an almost-universal vertex. The main result is as follows.

\begin{lemma}\label{lem:almost}
Let $G$ be an $n$-vertex graph with $e(G)\le2n-3$.
Let $v\in V(G)$ have exactly $t$ non-neighbours.
If
\[
 n\ge14t+8,
\]
then $\mu(G)\le4$.
\end{lemma}

\begin{proof}
Let $G$ be an $n$-vertex graph with $e(G) \le 2n-3$, and let $v $ be a vertex in $V(G)$ with  $d_G(v) = n-1-t$. Set $ H=G-v$. Then 
\begin{equation}\label{eq:eH-bound}
 e(H)\le (2n-3)-(n-1-t)=n-2+t.
\end{equation}

Suppose first that $H$ has a vertex $u$ with $d_H(u)=r\le 1$. 
Let
\[
 F=G-\{u,v\},\qquad N=n-2.
\]
Note that the number of edges of $G$ incident with at least one of $u,v$ is $n-1-t+r$. Thus, 
\[
 e(F)\le (2n-3)-(n-1-t+r)=n-2+t-r=N+t-r.
\]
Since $N=n-2\ge14t+6$, we have 
\[
 e(F)\le N+t-r\le N+\frac{N-6}{14}
 <\left\lfloor\frac{3N}{2}\right\rfloor.
\]
By Theorem \ref{CGM-THM-f-N,2}, $F$ has a relabelling with at most one common edge.
Extend this relabelling to $G$ by swapping $u$ and $v$. 
Outside $F$, the edge $uv$ (if it exists) contributes at most one common edge.
\rev{The $r$ edges from $u$ to $F$ can contribute at most $r$ common edges in each direction, so they contribute at most $2r$ further common edges.}
Thus at most $1+2r\le3$ common edges meet $\{u,v\}$, and hence the total number of common edges for $G$ is at most $4$.

In the following, suppose that $\delta(H)\ge2$. 
Let 
\[
 S=\{x\in V(H):d_H(x)\ge3\},\qquad s=|S|.
\]
It follows from $\delta(H)\ge2$ and  \eqref{eq:eH-bound} that 
\[
0\le  \sum_{x\in V(H)}(d_H(x)-2)\le2t-2.
\]
Consequently, 
\begin{equation}\label{eq:S-bounds}
 s\le2t-2,\qquad
 \sum_{x\in S}d_H(x)\le6t-6,\qquad
 \Delta(H)\le2t, 
\end{equation}
where the second bound follows as 
\[
 \sum_{x\in S}d_H(x)=2s+\sum_{x\in S}(d_H(x)-2)
 \le2(2t-2)+(2t-2)=6t-6.
\]

Let $T$ be the set of the $t$ non-neighbours of $v$. Then 
\[
 |T\cup S\cup N_H(S)|\le t+(2t-2)+(6t-6)=9t-8<n-1. 
\]
Choose 
\[
 u\in V(H)\setminus\bigl(T\cup S\cup N_H(S)\bigr).
\]
Then $uv\in E(G)$ and  $d_H(u)=2$. 
Write $N_H(u)=\{y_1,y_2\}$\rev{.}
It follows from  $u\notin N_H(S)$ that \rev{neither $y_1$ nor $y_2$ lies in $S$}.

Fix a vertex $w_0 \in T$, and set $F = G - \{u,v\}$ and $S_0 = S \cup \{w_0\}$. We claim that there exists an independent set $T_0 \subseteq V(F) \setminus S$ such that
\[
|T_0| = |S_0|, \qquad y_1 \in T_0.
\]
Indeed, the graph $F - S$ has maximum degree at most two, and $|V(F) \setminus S| \ge n - 2t$. Delete $y_1$ together with its neighbourhood in $F - S$. At most three vertices are removed, and the remaining graph still has maximum degree at most two. By Theorem~\ref{thm:greedy-alpha}, the graph $F - S$ contains an independent set containing $y_1$ of size at least
\[
1 + \frac{n - 2t - 3}{3} \ge 2t - 1 \ge |S_0|.
\]
The last inequality follows from $|S_0| \le s + 1 \le 2t - 1$, as given in \eqref{eq:S-bounds}. Choose a bijection $\psi_0:S_0\to T_0$ with $\psi_0(w_0)=y_1$.

Next consider the graph $F_0=F-\bigl(S_0\cup N_F(S_0)\bigr)$. If $w_0\notin S$, then $d_F(w_0)\le2$; if $w_0\in S$, its neighbourhood is already included in $N_F(S)$.
Hence, by \eqref{eq:S-bounds},  in either case, we have 
\[
 |S_0\cup N_F(S_0)|\le |S_0|+\sum_{x\in S}d_F(x)+2
 \le (s+1)+(6t-6)+2
 \le8t-5.
\]
This implies that
\[
 |V(F_0)|\ge(n-2)-(8t-5)=n-8t+3.
\]
Applying Theorem~\ref{thm:greedy-alpha} to $F_0$ with $\Delta(F_0)\le 2$ \rev{gives}
\[
 \alpha(F_0)\ge\frac{|V(F_0)|}{3}
 \ge\frac{n-8t+3}{3}>2t-2\ge s,
\]
where the strict inequality follows from $n\ge14t+8$.
Thus $F_0$ contains an independent set $R$ of size $s$.
Fix an arbitrary bijection $\psi_1:R\to S$ and combine \rev{$\psi_0$ and $\psi_1$} into a partial bijection $\psi$.  Observe that $\psi$ creates no common edge.
Edges inside $S_0$ are sent into the independent set $T_0$; there are no edges inside $R$; and there are no edges between $S_0$ and $R$ because $R\subseteq F_0$.
Thus every source edge with both endpoints already prescribed either has both images in an independent set, or has one endpoint in $S_0$ and the other in a set deliberately chosen outside $N_F(S_0)$.

We now apply Theorem~\ref{thm:list} to the unassigned vertices.
Let
\[
 U_1=V(F)\setminus(S_0\cup R),\qquad
 U_2=V(F)\setminus(T_0\cup S). 
\]
Then $U_1$ is the remaining source side and $U_2$ is the remaining target side.
Put $n_{\mathrm{rem}}=|U_1|=|U_2|$.
Define the graph triple
\[
 G_1=F[U_1],\qquad G_2=F[U_2],
\]
and let $G_3$ be the bipartite graph with parts $U_1$ and $U_2$ in which a pair $(x,y)\in U_1\times U_2$ is an edge if and only if there is a prescribed source vertex $z\in S_0\cup R$ such that
\[
 zx\in E(F),\qquad \psi(z)y\in E(F).
\]
Thus $G_1$ and $G_2$ are the two white graphs in Theorem~\ref{thm:list}, and $G_3$ is the yellow graph recording exactly the forbidden images forced by edges crossing from the prescribed source set to the remaining source vertices.

We verify the maximum-degree hypothesis of Theorem~\ref{thm:list}. Recall that $S$ has been removed from both the source and target copies. This implies 
\begin{equation}\label{EQ:max-degree-G1-G2}
  \Delta(G_1)\le2,\qquad \Delta(G_2)\le2.   
\end{equation}

Next we claim that
\begin{equation}\label{EQ:max-degree-G3}
 \Delta(G_3)\le4t.  
\end{equation}
Indeed, for a source vertex $x\in U_1$, a forbidden target $y$ can only be produced by a prescribed neighbour $z\in N_F(x)\cap(S_0\cup R)$.
Since $x\notin S$, there are at most two such prescribed neighbours $z$.
For each one of them, the possible forbidden targets $y$ lie in $N_F(\psi(z))$, and $d_F(\psi(z))\le2t$ by \eqref{eq:S-bounds}.
Hence the source-side degree of $x$ in $G_3$ is at most $4t$.
The target-side estimate is the same.

Since $|S_0|\le s+1$ and $s\le2t-2$,
\begin{equation}\label{EQ:lower-bound-n-rem}
 n_{\mathrm{rem}}=(n-2)-(|S_0|+s)
 \ge(n-2)-(2s+1)
 \ge n-4t+1
 \ge10t+9.
\end{equation}
Combining \eqref{EQ:max-degree-G1-G2}, \eqref{EQ:max-degree-G3} and \eqref{EQ:lower-bound-n-rem}, \rev{we have} 
\[
 \Delta(G_1)\Delta(G_2)+\Delta(G_3)
 \le 2\cdot2+4t<\frac{n_{\mathrm{rem}}}{2}.
\]
Applying Theorem~\ref{thm:list} gives a list packing $\phi:U_1\to U_2$ of the triple $(G_1,G_2,G_3)$. 
Combining $\phi$ with the prescribed bijection $\psi$ gives a bijection $\pi:V(F)\to V(F)$.
By the definition of $G_3$ and the extension principle from Section~\ref{Preliminaries}, we have 
\[
 E(F)\cap E(F_{\pi})=\emptyset,
 \qquad \pi(w_0)=y_1.
\]

\rev{Finally,} define a permutation of $V(G)$ by swapping $u$ and $v$ and using $\pi$ on $F$.
For the number of common edges, note that $uv$ contributes one.
The two edges from $u$ into $H$ contribute at most two.
These are the possible common edges coming from source edges incident with $u$: after the swap they become edges incident with $v$, and there are only the two such source edges from $u$ into $H$.
In the reverse direction, an edge containing $v$ can become \rev{an} edge containing $u$ only through a preimage of $y_1$ or $y_2$; the preimage of $y_1$ is $w_0$, and $vw_0\notin E(G)$.
Indeed, after the swap the only neighbours of $u$ outside $\{v\}$ in the target position are $y_1,y_2$, so a source edge incident with $v$ can survive as an edge incident with $u$ only if its other endpoint is mapped to one of these two vertices.
Thus this direction contributes at most one additional edge.
Therefore there are at most $4$ common edges. This completes the proof of Lemma \ref{lem:almost}. 
\end{proof}

\begin{corollary}\label{cor:nonneighbors}
For all sufficiently large $n$, if $G$ is an $n$-vertex graph with $e(G)\le2n-3$ and $\mu(G)\ge5$, then every vertex has at least $n/15$ non-neighbours.
\end{corollary}

\section{Proof of Theorem~\ref{thm:main}}

\rev{In this section, we prove the lower bound on $f(n,5)$ by contradiction.}

\begin{theorem}\label{thm:lower}
For all sufficiently large $n$,
\[
 f(n,5)\ge2n-2.
\]
\end{theorem}

Note that combining Theorem~\ref{thm:lower} with \rev{Lemma}~\ref{prop:upper} gives Theorem~\ref{thm:main}.\\

We prove Theorem~\ref{thm:lower} \rev{using the following strategy: from} an infinite sequence of counterexamples we extract a limiting degree profile, choose a fixed finite set $C$ of high-degree core vertices, find low-degree buffer vertices into which the core can be moved, and then absorb all neighbours of those buffers into another bounded set.

Suppose for contradiction that there is an infinite sequence of counterexamples $G_n$ \rev{(after passing to a subsequence, we continue to index the graphs by their orders $n$)} satisfying
\begin{equation}\label{eq:counterexamples}
 e(G_n)\le2n-3,\qquad \mu(G_n)\ge5.
\end{equation}
Write the degree sequence of $G_n$ in non-increasing order as
\[
 d_{n,1}\ge d_{n,2}\ge\cdots\ge d_{n,n}.
\]
By a diagonal subsequence argument we may assume that, for every fixed $i$,
\[
 \frac{d_{n,i}}n\longrightarrow\alpha_i.
\]

The sequence $(\alpha_i)$ is non-increasing and non-negative.
For every fixed $r$, by \eqref{eq:counterexamples}, 
\[
 \sum_{i=1}^r\alpha_i
 =\lim_{n\to\infty}\frac1n\sum_{i=1}^r d_{n,i}
 \le\limsup_{n\to\infty}\frac{4n-6}{n}=4.
\]
Letting $r\to\infty$ gives
\begin{equation*}
 \sum_{i\ge1}\alpha_i\le4.
\end{equation*}
This implies that 
\begin{equation}\label{eq:alpha-tail}
 \alpha_i\to0,\qquad k\alpha_{k+1}\to0.
\end{equation}

Fix $D_0=10^3$ and let $\varepsilon_0\le 10^{-3}$ be the constant in Lemma~\ref{lem:sparse}. We have $\Delta(G_n)\ge \varepsilon_0 n$ \rev{for} all sufficiently large $n$, since otherwise, applying Lemma~\ref{lem:sparse} to two copies of $G_n$ with no forbidden pairs gives \rev{a} self-packing for all sufficiently large $n$, contradicting \eqref{eq:counterexamples}.
Hence, we can set 
\begin{equation*}
 a:=\alpha_1>0.
\end{equation*}
Put
\[
 \gamma:=\frac1{15}-\frac4{D_0+1}>0.
\]
By \eqref{eq:alpha-tail}, choose a fixed $k\ge 1$ so  that
\begin{equation*}
 k\alpha_{k+1}<\min\left\{\frac a{100},\frac\gamma{100},\frac{\varepsilon_0}{100}\right\}.
\end{equation*}
We may further fix a constant $\eta>\alpha_{k+1}$ such that
\begin{equation}\label{eq:eta-choice}
 \eta<\frac{\varepsilon_0}{10},\qquad
 k\eta<\min\left\{\frac a{100},\frac\gamma{100},\frac{\varepsilon_0}{100}\right\}.
\end{equation}

Let $C$ be the set of the $k$ \rev{highest-degree} vertices of $G_n$.  We call $C$ the \emph{core} and its elements the \emph{core vertices}. Set
\[
 D=d_{n,k+1},\qquad S_C=\sum_{c\in C}d(c). 
\]
\rev{Then,} letting $n\to\infty$ gives
\begin{equation}\label{eq:core-limits}
 \frac Dn\longrightarrow\alpha_{k+1},\qquad
 \frac{S_C}{n}\longrightarrow\sum_{i=1}^k\alpha_i\ge a.
\end{equation}
In particular, after increasing the lower bound on $n$ if necessary, we may assume throughout the remainder of the proof that
\begin{equation}\label{eq:D-bound}
 D\le\eta n.
\end{equation}

Define the \emph{reservoir}
\[
 \mathcal B=\{v\notin C:d(v)\le3,\ d(v,C)\le2\}.
\]

\begin{lemma}\label{lem:reservoir}
For every $n$,
\[
 |\mathcal B|\ge\frac{S_C}{6}-2k.
\]
\end{lemma}

\begin{proof}
It follows from  $e(G)\le2n-3$ that 
\begin{equation}\label{eq:deficit}
 \sum_{v\in V(G)}(4-d(v))=4n-2e(G)\ge6.
\end{equation}
Let
\[
 P_+=\sum_{d(v)\ge5}(d(v)-4),\qquad
 M_-=\sum_{d(v)\le3}(4-d(v)).
\]
By~\eqref{eq:deficit}, we have $M_-\ge P_++6$.
For each $c\in C$, the contribution of $c$ to $P_+$ is at least $d(c)-4$, and therefore
\[
 P_+\ge S_C-4k.
\]
The contribution to $M_-$ from vertices of $C$ is at most $4k$.
\rev{Consequently,} the total deficit over vertices outside $C$ of degree at most three is at least
\[
 S_C-8k+6.
\]
Let $g=|\mathcal B|$, and let $b$ be the number of $3$-vertices outside $C$ all three of whose neighbours lie in $C$.
Each of the $b$ exceptional vertices contributes exactly $1$ to the deficit, while each vertex of $\mathcal B$ contributes at most $4$.
Hence
\[
 b+4g\ge S_C-8k+6.
\]
On the other hand, the $b$ exceptional vertices contribute $3b$ edges from $C$ to $V(G)\setminus C$, so $3b\le S_C$. We conclude that 
\[
 4g\ge\frac{2S_C}{3}-8k+6,
\]
which is stronger than the claimed bound.
\end{proof}

By Lemma~\ref{lem:reservoir} and \eqref{eq:core-limits},
\[
 \liminf_{n\to\infty}\frac{|\mathcal B|}{n}
 \ge\frac16\sum_{i=1}^k\alpha_i
 \ge\frac a6.
\]
Hence, for all sufficiently large $n$,
\[
 |\mathcal B|\ge\frac{an}{7}.
\]

We select a \emph{buffer set}
\[
 B=\{b_1,\ldots,b_k\}\subseteq\mathcal B
\]
with the following properties:
\begin{enumerate}[label=(\roman*)]
\item $B$ is independent;
\item no two vertices of $B$ have a common neighbour outside $C$.
\end{enumerate}

Such a choice is possible greedily. At any stage of the greedy procedure, call a vertex $v\in\mathcal B$ available if it has not yet been selected and, for every previously selected vertex $v_j$, we have
\[
vv_j\notin E(G)
\quad\text{and}\quad
(N(v)\cap N(v_j))\setminus C=\emptyset.
\]
Initially, every vertex of $\mathcal B$ is available. After choosing an available vertex $v$, we delete $v$ itself and all available neighbours of $v$; the latter number is at most $d(v)\le3$. This preserves property (i). We also delete every available vertex \rev{that shares a common neighbour outside $C$ with $v$}. Indeed, $v$ has at most three neighbours in $(V(G)\setminus C)$, and each such neighbour has degree at most $D$. Thus this second deletion removes at most $3D$ available vertices and preserves property (ii). Consequently, each step removes at most $1+3+3D=3D+4$ vertices from the available set. 
By  \eqref{eq:eta-choice} and \eqref{eq:D-bound}, all $k$ choices together remove at most
\[
 k(3D+4)\le3k\eta n+4k<\frac{an}{10}<|\mathcal B|
\]
vertices. Thus the greedy construction succeeds.

The elements \rev{of} $B$ are called the \emph{buffer vertices}. They will serve as the target positions for the core vertices. Now we define bijections between $B$ and $C$. By the choice of $B$, we have 
\begin{equation}\label{eq:q-bound}
 q:=e(C,B)\le2k.
\end{equation}
Choose \rev{independent uniformly random} bijections
\[
 \alpha:C\to B,\qquad \beta:B\to C.
\]
For a fixed edge $cb\in E(C,B)$ (we call $cb$ a \emph{$C$--$B$ edge}),  the probability that the image edge $\alpha(c)\beta(b)$ is again a $C$--$B$ edge is $q/k^2$. By the linearity of expectation and \eqref{eq:q-bound}, the expected number of common $C$--$B$ edges is $q^2/k^2\le4$.
Fix $\alpha,\beta$ for which this number is at most four.
Because $B$ is independent, $C$--$C$ edges are sent to $B$--$B$ non-edges.
Thus the partial permutation
\begin{equation}\label{eq:partial-core-buffer}
 c\mapsto\alpha(c)\quad \text{for } c\in C,\qquad
 b\mapsto\beta(b)\quad \text{for } b\in B
\end{equation}
has at most four common edges inside $C\cup B$.

Note that the buffer vertices may have neighbours outside the core. \rev{We now choose two further sets to absorb the external neighbours of the buffer vertices.}

Let $X=N(B)\setminus C$\rev{.} \rev{By the choice of $B$, we have}
\begin{equation*}
 |X|\le3k. 
\end{equation*}
Moreover, by \rev{property} (ii) of $B$, each $x\in X$ has a unique neighbour $b(x)\in B$.
Define
\[
 c^-(x)=\alpha^{-1}(b(x)),\qquad c^+(x)=\beta(b(x)).
\]

We now choose pairwise disjoint sets
\[
 P=\{p_x:x\in X\},\qquad Q=\{q_x:x\in X\}
\]
outside $C\cup B\cup X$ so that

\begin{enumerate}[label=(\roman*)]
\item \label{eq:PQ-degree} $d(p_x),d(q_x)\le D_0$; 
\item  \label{eq:PQ-nonneighbor} $p_x \notin N(c^+(x)),\qquad q_x\notin N(c^-(x))$;
\item $P$ and $Q$ are independent;
\item \label{eq:PQ-edgefree} $E(X,P)=E(X,Q)=E(P,Q)=\emptyset$\rev{.} 
\end{enumerate}
Such a choice is possible greedily.
Fix any order in which the names $p_x$ and $q_x$, $x\in X$, will be assigned \rev{to} actual vertices of $G$.
At a given stage, let $Z$ be the set of vertices already chosen.
If the next name to be assigned is $p_x$, start from
\[
 A=\{v\in V(G): d(v)\le D_0,\ v\notin N(c^+(x))\};
\]
if the next name to be assigned is $q_x$, start instead from
\[
 A=\{v\in V(G): d(v)\le D_0,\ v\notin N(c^-(x))\}.
\]
Since $e(G_n)\le 2n-3$, the  number of vertices in $G_n$ of degree greater than $D_0$ is at most
\[
 \frac{4n-6}{D_0+1}.
\]
\rev{Together with Corollary~\ref{cor:nonneighbors}, this shows that} every vertex in $C$ has more than 
\[
 \frac n{15}-\frac{4n-6}{D_0+1}>\gamma n
\]
non-neighbours of degree at most $D_0$. This implies that 
\[
|A|>\gamma n
\]
in either case.

Call a vertex $v\in A$ \emph{available} at this stage if
\[
 v\notin C\cup B\cup X,\qquad
 v\notin N(X),\qquad
 v\notin \{z\}\cup N(z)\quad\text{for every }z\in Z.
\]
Let $\mathrm{Av}$ denote the set of available vertices at this stage.
If an available vertex is chosen, it is assigned to the current name. It remains to check that $\mathrm{Av}$ is non-empty at every stage.
Among the more than $\gamma n$ vertices of $A$, the deletion of $C\cup B\cup X$ removes at most
\[
 |C|+|B|+|X|\le5k
\]
candidates.
The deletion of $N(X)$ removes at most
\[
 |N(X)|\le\sum_{x\in X}d(x)\le |X|D\le3kD\le3k\eta n. 
\]
Finally, each previously chosen vertex $z\in Z$ was selected from a candidate set of the above form, so $d(z)\le D_0$.
Deleting $z$ together with its neighbourhood removes at most $D_0+1$ further candidates.
Since $|Z|<2|X|\le6k$, all such deletions together remove at most $6k(D_0+1)$ candidates.
Consequently
\begin{equation}\label{eq:available-candidates}
 |\mathrm{Av}|
 >\gamma n-3k\eta n-5k-6k(D_0+1).
\end{equation}
By \eqref{eq:eta-choice}, $3k\eta<3\gamma/100$.
Hence the right-hand side of \eqref{eq:available-candidates} is larger than
\[
 \frac{97\gamma}{100}n-5k-6k(D_0+1),
\]
which is positive for all sufficiently large $n$.
Thus an available vertex exists at every stage. The preceding availability conditions give \rev{properties} \ref{eq:PQ-degree}--\ref{eq:PQ-edgefree}.

Extend the partial permutation \eqref{eq:partial-core-buffer} to a permutation of
\[
 S=C\cup B\cup X\cup P\cup Q
\]
by the three-cycles
\begin{equation}\label{eq:three-cycles}
 x\mapsto p_x\mapsto q_x\mapsto x\qquad(x\in X).
\end{equation}
Let $\pi_S$ denote the map on $S$ defined by \eqref{eq:partial-core-buffer} and \eqref{eq:three-cycles}.

\begin{lemma}\label{lem:S}
The permutation $\pi_S$ creates at most four common edges with both endpoints in $S$.
\end{lemma}

\begin{proof}
We shall show that every common edge of $G[S]$ under $\pi_S$ is already one of the $C$--$B$ common edges counted when $\alpha$ and $\beta$ were chosen.

We use the partition
\[
 S=C\cup B\cup X\cup P\cup Q
\]
and inspect the possible types of the source edge $uv$.
Some types are excluded before applying \rev{$\pi_S$}.
There are no edges inside $B$, $P$, or $Q$, as these sets are independent.
By \ref{eq:PQ-edgefree}, there are no edges of types $X$--$P$, $X$--$Q$, or $P$--$Q$.
Also there are no edges from $B$ to $P\cup Q$.

We next eliminate all remaining source types other than $C$--$B$.
An edge inside $C$ (or $X$) is mapped into $B$ (or $P$), hence to a non-edge.
If the source edge has type $C$--$X$ (or $C$--$P$), then its image has type $B$--$P$ (or $B$--$Q$), and there are no such edges.

It remains to consider the two types in which the special choices of $p_x$ and $q_x$ are used.
First suppose that $cq_x\in E(G)$ with $c\in C$ and $x\in X$.
Its image is $\alpha(c)x$.
The vertex $x$ has the unique neighbour $b(x)$ in $B$.
Thus, if $\alpha(c)x$ were an edge, then $\alpha(c)=b(x)$, so
\[
 c=\alpha^{-1}(b(x))=c^-(x),
\]
contrary to $q_x\notin N(c^-(x))$ from \ref{eq:PQ-nonneighbor}.
Hence no source edge of type $C$--$Q$ is common.  Suppose that $bx\in E(G)$ with $b\in B$ and $x\in X$.
By the uniqueness of the $B$-neighbour of $x$, we have $b=b(x)$.
The image of $bx$ is therefore
\[
 \beta(b(x))p_x=c^+(x)p_x,
\]
which is a non-edge by \ref{eq:PQ-nonneighbor}.
Hence no source edge of type $B$--$X$ is common.

The above cases exhaust all source edges in $G[S]$ except the $C$--$B$ edges.
Therefore at most the four previously counted $C$--$B$ edges can remain common under \rev{$\pi_S$}.
\end{proof}

After the core, buffers, and their neighbours have been prescribed, the rest of the proof is an application of Lemma~\ref{lem:sparse}. \rev{For simplicity, write $G:=G_n$.}
Let $U=V(G)\setminus S$. \rev{We build the following graph triple.}
Let $G_U^{\src}$ be a source copy of $G[U]$, and let $G_U^{\tgt}$ be a target copy of $G[U]$.
The third graph, denoted $F_U$, is the bipartite forbidden-pair graph between the source and target copies of $U$.
For a source vertex $u\in U$ and a target vertex $v\in U$, put $uv\in E(F_U)$ if and only if there is some $s\in S$ such that
\begin{equation*}
 su\in E(G),\qquad \pi_S(s)v\in E(G).
\end{equation*}
\rev{Observe} that a forbidden pair is exactly a choice $u\mapsto v$ which would create a common edge crossing the cut $(S,U)$ together with the already fixed map on $S$.

We verify the hypotheses of Lemma~\ref{lem:sparse} for the triple $(G_U^{\src},G_U^{\tgt},F_U)$\rev{.}
First,
\[
 |S|=|C|+|B|+|X|+|P|+|Q|\le11k,
\]
so $|U|\ge n-11k$. For \rev{sufficiently} large $n$ with $n>33k$, we have 
\begin{equation}\label{eq:U-edges}
 e(G_U^{\src})=e(G_U^{\tgt})=e(G[U])
 \le2n-3<3(n-11k)\le3|U|.
\end{equation}
Thus the edge hypothesis of Lemma~\ref{lem:sparse} holds.

Next, it follows from  $U\cap C=\emptyset$ and \eqref{eq:D-bound} that 
\begin{equation*}
 \Delta(G_U^{\src})=\Delta(G_U^{\tgt})=\Delta(G[U])\le D\le \eta n.
\end{equation*}
Using $|U|\ge n-11k$ gives 
\begin{equation}\label{EQ:maximum-degree-GU}
 \Delta(G_U^{\src}),\Delta(G_U^{\tgt})<\varepsilon_0|U|,     
\end{equation}
for all sufficiently large $n$.

It remains to bound the maximum degree of the forbidden graph $F_U$.
Note that for each $b\in B$, 
\begin{equation}\label{eq:B-no-U-neighbor}
 d_U(b)=0.
\end{equation}

We first bound the source-side degrees of $F_U$.
Fix $u\in U$ in the source copy.
Every forbidden target $v$ has a witness $s\in N_G(u)\cap S$, and once $s$ is fixed, the possible targets $v$ lie in $N_G(\pi_S(s))\cap U$.
Hence the degree of $u$ in $F_U$ is at most
\[
 d_{F_U}^{\src}(u)\le\sum_{s\in N_G(u)\cap S}d_U(\pi_S(s)).
\]
If $s\in C$, then $\pi_S(s)\in B$, so the summand is zero by \eqref{eq:B-no-U-neighbor}.
If $s\in B$, then $d_U(s)=0$, so such an $s$ is not adjacent to  $u$.
For $s\in X$, the image $\pi_S(s)$ lies in $P$ and has degree at most $D_0$.
For $s\in P$, the image lies in $Q$ and again has degree at most $D_0$.
For $s\in Q$, the image lies in $X$, whose vertices have degree at most $D$.
Using $|X|=|P|=|Q|\le3k$ gives
\begin{equation}\label{eq:F-source}
 d_{F_U}^{\src}(u)\le3kD+6kD_0.
\end{equation}

The target-side bound is analogous but slightly different in its bookkeeping.
Fix $v\in U$ in the target copy.
If a source vertex $u$ is forbidden for $v$, then some $s\in S$ satisfies $v\in N_G(\pi_S(s))$, and the possible sources $u$ lie in $N_G(s)\cap U$.
Thus
\[
 d_{F_U}^{\tgt}(v)\le\sum_{\substack{s\in S\\v\in N_G(\pi_S(s))}}d_U(s).
\]
For $s\in C$, the condition $v\in N_G(\pi_S(s))$ is impossible because $\pi_S(s)\in B$ has no neighbour in $U$.
For $s\in B$, $d_U(s)=0$.
For $s\in X$ we have $d_U(s)\le D$, while for $s\in P\cup Q$ we have $d_U(s)\le D_0$.
Again using $|X|=|P|=|Q|\le3k$ yields
\begin{equation}\label{eq:F-target}
 d_{F_U}^{\tgt}(v)\le3kD+6kD_0.
\end{equation}
Combining \eqref{eq:F-source} and \eqref{eq:F-target}, we have 
\begin{equation*}
 \Delta(F_U)\le3kD+6kD_0.
\end{equation*}
This together with \eqref{eq:D-bound} \rev{yields} that 
\[
 \frac{\Delta(F_U)}{|U|}
 \le\frac{3k\eta n+6kD_0}{n-11k}
 \longrightarrow3k\eta<\frac{3\varepsilon_0}{100}<\varepsilon_0.
\]
Consequently, for all sufficiently large $n$, we have 
\begin{equation}\label{eq:maximum-dgree-FU}
 \Delta(F_U)<\varepsilon_0|U|.
\end{equation}

By \eqref{eq:U-edges}, \eqref{EQ:maximum-degree-GU} and \eqref{eq:maximum-dgree-FU}, 
we have verified all hypotheses of Lemma~\ref{lem:sparse}, with $|U|$ in place of $n$.
There is a bijection $\phi:U\to U$ such that
\begin{align}
 uv\in E(G[U])&\Longrightarrow \phi(u)\phi(v)\notin E(G[U]), \label{eq:U-white-avoid}\\
 u\phi(u)&\notin E(F_U)\qquad(u\in U). \label{eq:U-yellow-avoid}
\end{align}
Define a permutation $\pi$ of $V(G)$ by
\[
 \pi(s)=\pi_S(s)\quad \text{for }s\in S,\qquad
 \pi(u)=\phi(u)\quad \text{for } u\in U.
\]
By \eqref{eq:U-white-avoid}, no common edge has both endpoints in $U$.
By \eqref{eq:U-yellow-avoid} and the definition of $F_U$, no common edge crosses between $S$ and $U$.
Thus, it follows from Lemma~\ref{lem:S} that there are at most four common edges, contradicting \eqref{eq:counterexamples}. This completes the proof of Theorem \ref{thm:lower}.

\section{Concluding remarks}
Our main result shows that
\[
 f(n,5)=2n-2
\]
for all sufficiently large $n\ge n_0$.  We do not attempt to determine or optimize the threshold $n_0$.  The problem originated with Mullin and was formulated by Erd\H{o}s as a minimum-intersection problem for isomorphic copies~\cite{Erdos}; it was later developed within graph packing and near packing~\cite{MRS,Eaton,Zak,KZ}.  Our core--buffer--completion method  uses list packing.  The same separation of a bounded exceptional set from a sparse remainder may be useful in constrained graph embeddings, near decompositions, and permutation problems with forbidden images.

For a general fixed $k$, one may similarly try to spend the allowed $k-1$ common edges inside a bounded core and pack the remainder edge-disjointly.  However, the deficit calculation around degree $4$, the low-degree reservoir, and the four-edge bookkeeping are specific to $k=5$.  Other values of $k$ require new extremal and structural input, and when $k$ grows with $n$, the bounded-core argument is no longer directly applicable.

The broader asymptotic picture is summarized by the following conjecture, communicated by Chung and Graham and recorded by Erd\H{o}s~\cite{Erdos}.

\begin{conjecture}[Chung and Graham]
\[
 f\left(n,\left\lfloor\frac{t(t-1)}4\right\rfloor\right)
 =\frac{tn}{2}+o(n).
\]
Moreover, the error term may perhaps be strengthened from $o(n)$ to $o(1)$.
\end{conjecture}

\section*{Declaration on the use of AI}
The proof strategy was developed by the authors, inspired by a work of Gy\H{o}ri, Kostochka, McConvey and Yager~\cite{GKMY}.  The authors used generative AI tools to help check and simplify routine technical computations in several auxiliary lemmas. All \rev{AI-assisted} suggestions and computations were independently checked and
verified by the authors.

\end{document}